\documentclass[12pt,fleqn,leqno,reqno]{amsart}

\usepackage{amsfonts,amsmath,amsthm}
\usepackage{amssymb,epsfig}
\usepackage{enumerate} 

\usepackage{color} 
\definecolor{vert}{rgb}{0,0.6,0}

\usepackage[text={425pt,650pt},centering]{geometry}
\usepackage{graphicx}
\usepackage{epsfig}
\usepackage{tikz,esvect}
\usetikzlibrary{3d,calc,shapes}

\usepackage{caption}
\usepackage{color} 
\definecolor{vert}{rgb}{0,0.6,0}

\numberwithin{figure}{section}

\theoremstyle{plain}
\newtheorem{thm}{Theorem}[section]

\newtheorem{quest}{Question}

\newtheorem{lem}[thm]{Lemma}
\newtheorem{cor}[thm]{Corollary}
\newtheorem{prop}[thm]{Proposition}
\theoremstyle{remark}
\newtheorem{rem}{\bf{Remark}}
\numberwithin{equation}{section}

\newcommand{\N}{\mathbb{N}}

\newcommand{\R}{\mathbb{R}}

\newcommand{\cA}{\mathcal{A}}
\newcommand{\cB}{\mathcal{B}}

\newcommand{\al}{\alpha}
\newcommand{\gam}{\gamma}
\newcommand{\del}{\delta}

\newcommand{\ol}{\overline}

\newcommand{\source}{X_0}
\newcommand{\target}{Z}

\begin{document}

\title[Selection problems in Large Deviations in Games]
{Selection problems in Large Deviations in Games under the logit choice protocol}

\author[H. V. TRAN]
{Hung V. Tran}

\dedicatory{In memory of William H. Sandholm}

\thanks{
The work of HT is partially supported by NSF grant DMS-1664424 and NSF CAREER grant DMS-1843320.
}

\address[H. V. Tran]
{
Department of Mathematics, 
University of Wisconsin Madison, Van Vleck Hall, 480 Lincoln Drive, Madison, Wisconsin 53706, USA}
\email{hung@math.wisc.edu}


\maketitle

\begin{abstract}
We study large deviations in coordination games under the logit choice protocol.
A major open question that \cite{SS, SS2} posed is whether large deviations properties under the small noise double limit and the large population double limit  are identical or not.
We rephrase this open question in the PDE language as some selection problems, and we provide some definitive answers to these problems.
\end{abstract}


\section{Introduction}

\subsection{The general setting}
We first give a brief discussion of optimal control problems with semilinear running costs and state constraints considered in \cite{SS, STA}.
Let $n,m \in \N$ be such that $m\leq n$.
Let $X \subset \R^n$ be a closed $m$-dimensional polytope. 
Let $TX$ denote the set of tangent vectors from points in the relative interior $X ^\circ$ of $X$.  
If $X$ is $n$-dimensional, then $TX = \R^n$; more generally, if the affine hull of $X$ is a translation of a subspace $Y \subset \R^n$, then $TX = Y$.  
In what follows, topological statements are made with respect to the relative topology on $X$; for example, we will refer to $X^\circ$ as the interior of $X$.  
Likewise, derivatives of functions $f \colon X\to\R^m$ will be understood as maps $Df \colon X \to \mathcal{L}(TX, \R^m)$ from $X$ to linear functions from $TX$ to $\R^m$.

Let $TX(x) = \{t(y-x) \colon y \in X, t \geq 0\} \subseteq TX$ be the tangent cone of $X$ at $x$.  
Here, $TX(x)$ is the set of feasible controls at state $x$, though as we note shortly, it will be enough to restrict attentions to controls in a compact subset of $TX(x)$.

We assume that the running cost function $L \colon X \times TX \to [0, +\infty]$ takes a semilinear form.  
Specifically, we assume that for a given Lipschitz continuous function $\Psi\colon X \to [0,\infty)^n$, we have
\begin{equation}\label{eq:RC}
L(x,v) = \begin{cases}
\displaystyle\sum_{i=1}^n \Psi_i(x)\,[v_i]_+\qquad&\text{if }v \in TX(x),\\
+\infty\qquad&\text{otherwise}.
\end{cases}
\end{equation}
Here, $[v_i]_+=\max\{v_i,0\}$ for $1\leq i \leq n$.
Thus, the constraint that the state remain in $X$ is built into the definition of running costs.
We also say that $L$ is our given Lagrangian.

For $T>0$, let $\Phi_T$ be the set of Lipschitz continuous paths $\phi\colon [0, T] \to X$, and let $\Phi = \bigcup_{T\ge0}\Phi_T$.  Then the \emph{path cost function} $c\colon \Phi \to \R_+$ is defined by
\begin{equation}\label{eq:PathCost}
c(\phi) = \int_0^T L(\phi(t), \dot\phi(t))\,d t\qquad\text{when }\phi \in \Phi_T.
\end{equation} 

The \emph{source problem} for a given compact set $\source \subset X$ is that of finding the minimal cost of reaching each state in $X$ from a free initial condition in $\source$.  The value function for the source problem for $\source$ is, for $x\in X$,
\begin{equation}\label{eq:SP}\tag{SP}
W(x) = \inf\:\{c(\phi) \colon \phi \in \Phi_T\text{ for some }T \geq 0, \phi(0) \in \source, \phi(T) = x \}.
\end{equation}
Likewise, the \emph{target problem} for a given compact set $Z \subset X$ is that of finding the minimal cost of reaching a free state in $Z$ from initial condition $x\in X$.  The value function for the target problem for $Z$ is
\begin{equation}\label{eq:TP}\tag{TP}
V(x) = \inf\:\{c(\phi) \colon \phi \in \Phi_T\text{ for some }T \geq 0,  \phi(0) = x, \phi(T) \in Z  \}.
\end{equation}

\begin{rem}\label{rem:Speed}
It follows immediately from the semilinearity of the running cost function \eqref{eq:RC} that the cost of a path does not depend on the speed at which it is traversed:  if $\phi \in \Phi_T$ and $\hat\phi \in \Phi_{\hat T}$ for some $T, \hat T>0$ differ only by a reparameterization of time, then $c(\phi) = c(\hat\phi)$.
Because of this, the solutions to problems \eqref{eq:SP} and \eqref{eq:TP} do not change if we restrict the control variable $u$ to a compact convex set whose conical hull is $TX$.  Viewed through the prism of feedback controls, semilinearity implies that we need only determine the optimal \emph{directions} of motion from each state; the speed of motion in an optimal direction is irrelevant.
\end{rem}

We take advantage of this property by introducing a convenient restriction on the control variable.  
Let $|\cdot|$ denote the $\ell_1$ norm on $\R^n$, so that $|u| = \sum_{i=1}^n |u_i|$.  
For $r > 0$, let $ B_r = \{u \in \R^n \colon |u|\leq r\}$ be the closed ball  of radius $r$ on $\R^n$.  
The foregoing discussion shows that  in solving \eqref{eq:SP} and \eqref{eq:TP}, there is no loss in restricting attention to paths $\phi\in\Phi$ with $\dot\phi(t) \in  B_r$ for almost all $t \geq 0$ for any fixed $r>0$.  
To take advantage of this, we replace the running cost function \eqref{eq:RC} with one in which controls outside of $ B_r$ are infeasible:
\begin{equation}\label{eq:RCNew}
L(x,v) = \begin{cases}
\displaystyle\sum_{i=1}^n \Psi_i(x)\,[v_i]_+\qquad&\text{if }v \in TX(x) \cap  B_r,\\
+\infty\qquad&\text{otherwise}.
\end{cases}
\end{equation}
For $v \in \R^n$, we define the componentwise positive part function $[v]_+$ by $([v_+])_i = [v_i]_+$, and we define $[v]_-$ analogously.  Using this notation, we can write the first case of \eqref{eq:RCNew} concisely as
\begin{equation}\label{eq:RC2}
L(x,v) =\Psi(x) \cdot [v]_+ \qquad\text{if } v\in TX(x)\cap  B_r.
\end{equation}
Set $Y=TX$.
Let $H \colon X \times Y \to \R$ denote the corresponding Hamiltonian
\begin{equation}\label{eq:Ham}
H(x, u) = \max_{v \in Y \cap  B_r} \left( u\cdot v - L(x,v) \right)=\max_{v \in Y \cap  B_r} \left( u\cdot v - \Psi(x) \cdot [v]_+ \right). 
\end{equation}
As usual, $H(x, \cdot)$ is the Legendre transform (convex conjugate) of $L(x, \cdot)$.
Let $\|{\Psi}\|_\infty = \max_{x \in X}|\Psi(x)|$ denote the $L^\infty$ norm of $\Psi$. 
Thanks to formula \eqref{eq:Ham}, $H(x, \cdot)$ satisfies the following linear lower bound:
\begin{equation}\label{eq:coercive}
H(x,u) \geq \frac{r}{n}|u| - r\|\Psi\|_{\infty}.
\end{equation}
This gives us that $H$ is uniformly coercive on $X$, that is,
\[
\lim_{|u| \to \infty} \min_{x\in X} H(x,u)=+\infty.
\]

Here are some of the main results obtained in \cite{STA}. We say that a function $V \colon X \to \R$ is \emph{maximal} with respect to given properties if for any other function $V_0\colon X \to \R$ satisfying the properties, we have $V_0 \leq V$ on $X$.  

Theorems \ref{thm:Source} and \ref{thm:Target} characterize the solutions to the source and target problems in terms of subsolutions in the almost everywhere sense.  These inequalities are only required to hold almost everywhere, and in particular need not be checked at boundary states $\partial X$, or at states where the 
candidate function is not differentiable. 

\begin{thm}\label{thm:Source}
The solution to the source problem \eqref{eq:SP} is the maximal Lipschitz continuous function $W\colon X \to \R$ satisfying
\begin{equation}\label{eq:SolSource}
\begin{cases}
H(x,DW(x)) \leq 0 \qquad&\text{for almost all }x \in X^\circ;\\
W(x^\ast) = 0\qquad&\text{for all }x^\ast \in \source.
\end{cases}
\end{equation}
\end{thm}

\begin{thm}\label{thm:Target}
The solution to the target problem \eqref{eq:TP} is the maximal Lipschitz continuous function $V\colon X \to \R$ satisfying
\begin{equation}\label{eq:SolTarget}
\begin{cases}
H(x,-DV(x)) \leq 0\qquad&\text{for almost all }x \in X^\circ;\\
V(y^\ast) = 0\qquad&\text{for all }y^\ast \in Z.
\end{cases}
\end{equation}
\end{thm}

We only state a verification theorem for the target problem; the statement for the source problem is similar.

\begin{thm}\label{thm:VerTarget}
Suppose that $V\colon X \to \R$ is Lipschitz continuous that satisfies \eqref{eq:SolTarget}, and $V(y^\ast) = 0$ for all $y^\ast\in Z$.  
In addition, suppose that for each $x \in X \setminus  Z $, there is a time $T>0$ and a path $\phi \in \Phi_T$ with $\phi(0) = x$ and $\phi(T) \in  Z $ such that $V(x) = c(\phi)$.  
Then, $V$ is the solution to \eqref{eq:TP}.
\end{thm}

Hamilton-Jacobi equations with state-constraint boundary conditions were first studied in \cite{Soner}.
See also \cite{CDL, AT}.
The cost function \eqref{eq:PathCost} here does not have a discount factor,  which results in the fact that \eqref{eq:SolSource}  and \eqref{eq:SolTarget} are not monotone in the unknowns.
In general,  \eqref{eq:SolSource}  and \eqref{eq:SolTarget} have many solutions; for example, $0$ is always a solution to both as $H(x,0)=0$ for all $x\in X$.
This naturally leads us to consider the notion of maximal solutions as stated in Theorems \ref{thm:Source} and \ref{thm:Target}.
Since $p \mapsto H(x,p)$ is convex and coercive, it is important emphasizing that, for a Lipschitz continuous function $V$,
$V$ is an almost everywhere subsolution to \eqref{eq:SolSource} if and only if $V$ is its viscosity subsolution (see \cite[Chapter 2]{Tr}).
We chose to state the subsolutions in Theorems \ref{thm:Source} and \ref{thm:Target} in the almost everywhere way to make the statements simpler.

The verification theorem (Theorem \ref{thm:VerTarget}) allows us to make some intelligent guesses to find explicit solutions and optimal paths for both \eqref{eq:SP} and \eqref{eq:TP} in some large deviations in evolutionary game theory problems \cite{STA}.
See also \cite{Ari} for some new applications.

\subsection{Large deviations in coordination games and the logit choice protocol}
We provide here a brief description of large deviations in coordination games and  the logit choice protocol (see \cite{Bl, S-book, SS, STA} for more details).

Let $e_1, e_2, \ldots, e_n$, the standard basis of $\R^n$, be $n$ given equilibria. 
Denote by
\[
X= \left\{x_1 e_1 + x_2 e_2 +\cdots+ x_n e_n\colon x_i \geq 0 \text{ for } 1 \leq i \leq n,  \ \sum_{i=1}^n x_i=1 \right\}.
\]
Then,
\[
Y=TX= \R^n_0 = \left\{u\in \R^n\colon  \sum_{i=1}^n u_i=0\right\}.
\]

Consider a class of Markov chains $\{X^{N,\eta}_k\}_{k=0}^\infty$ parametrized by a population of size $N \in \N$ and a noise level $\eta>0$, which run on discrete grids $\mathcal X^N$ of mesh size $\frac{1}{N}$ in the simplex $X$.
These Markov chains describe the evolution of aggregate behavior in the given population of $N$ strategically interacting agents. 
Each agent adjusts his/her actions over time by following a noisy best response rule, under which the probabilities of choosing suboptimal actions
vanish at exponential rates in $\frac{1}{\eta}$.
Here, each agent chooses action from the common finite action set $\cA=\{1,2,\ldots, n\}$.
The population's aggregate behavior is described by a population state $x \in \mathcal X^N \subset X$, with $x=\sum_{i=1}^n x_i e_i$, where $x_i$ represents the fraction of the population playing strategy $i$.
In the class of games to be discussed, the Markov chains typically approach and then remain near pure states $e_1,\ldots, e_n$ corresponding to strict Nash equilibria, but the ergodicity of these processes ensure that transitions between such states must occur. 
Large deviations results were developed in \cite{SS} to describe the waiting times until and likely paths of transitions between those strict Nash equilibria.
Then, the analysis in \cite{SS, STA} concerns the {\it small noise double limit}, meaning that the noise level $\eta$ is first taken to zero, and then the population size $N$ to infinity.
Large deviation properties of $\{X^{N,\eta}_k\}_{k=0}^\infty$ are described in terms of solutions to optimal control problems with semilinear running costs and state constraints, and these are precisely where Theorems \ref{thm:Source}--\ref{thm:VerTarget} play essential roles.

\medskip

Let us be more specific.
Consider $A \in \R^{n\times n}$, and we use superscripts to refer to rows of $A$, and subscripts to refer to its columns.
More precisely, $A^i$ is the $i$-th row of $A$, $A_j$ is the $j$-th column of $A$, and $A^i_j$ is the $(i,j)$-th entry.
Denote by
\[
A^{i-j} = A^i - A^j = (e_i -e_j)'A, \qquad A^{i-j}_{k-l} = A^i_k - A^i_l - A^j_k + A^j_l = (e_i -e_j)' A(e_k-e_l).
\]
Agents are matched against all opponents to play a symmetric two-player normal form game $A$ given here, with $A^i_j$ is the payoff that an agent playing $i$ obtains when matched against an agent playing $j$.
During such a matching, the payoff obtained by the action $i$ player is $\sum_j A^i_j x_j=A^i x$.
State $x^\ast \in X$ is a {\it Nash equilibrium} of $A$ if all strategies in use at $x^\ast$ are optimal, that is,
\[
A^i x^\ast = \max_{1 \leq j \leq n } A^j x^\ast \qquad \text{ whenever } x_i^\ast >0.
\]
Nash equilibria can be characterized in terms of best-response regions.
For $1\leq i \leq n$, the best-response region for strategy $i$ is defined as
\[
\mathcal{B}^i=\left\{x\in X \colon A^{i-j}x = A^ix - A^jx \geq 0 \qquad \text{ for all } 1\leq j \leq n \right\},
\]
in which action $i$ is optimal.
The set $\mathcal{B}^{ij} = \mathcal{B}^i \cap \mathcal{B}^j$ is the boundary between the best-response regions for strategies $i$ and $j$.
It is clear that $x^\ast$ is a Nash equilibrium if $x^\ast \in \cB^i$ whenever $x_i^\ast >0$.
We study the normal form game $A$, which is always a {\it coordination game}, that is,
\[
A^i_i >A^j_i \qquad \text{ for } i \neq j.
\]
This means that if the opponent plays $i$, then it is best playing $i$ as well.
In coordination games, each pure state $e_i$ is a Nash equilibrium for $1\leq i \leq n$.

In the discrete stochastic model, each agent randomly receives some chances to revise their actions by applying a noisy best response protocol $\sigma^\eta: \R^n \to X^\circ$ with noise level $\eta>0$, a function that maps payoffs vectors in $\R^n$ to probabilities of choosing each action.
For a payoffs vector $\pi \in \R^n$, the probability of choosing to play action $j$ is given by $\sigma^\eta_j(\pi)$.
The {\it logit choice protocol} (see \cite{Bl}) is given by
\[
\sigma^\eta_j(\pi) = \frac{e^{\frac{\pi_j}{\eta}}}{\sum_{k=1}^n e^{\frac{\pi_k}{\eta}}}
\]
The suboptimal actions surely have vanishing probabilities as noise level $\eta \to 0$, which are captured by the unlikelihood function $\Upsilon: \R^n \to [0,\infty)^n$ as
\[
\Upsilon_j(\pi) =-\lim_{\eta \to 0} \eta \log \sigma^\eta_j(\pi)= \max_{1\leq i \leq n} \pi_i - \pi_j,
\]
which is piecewise linear.

Let
\[
\Psi(x) = Ax \qquad \text{ for all } x\in X.
\]
Then $\Upsilon_j(Ax)= \max_{1\leq i \leq n} A^i x -A^j x$.
In particular, for $x\in\cB^i$, $\Upsilon_i(Ax)=0$, and $\Upsilon_j(Ax)=A^{i-j}x$.
As computed in \cite{SS, STA}, the Hamiltonian $H: X \times Y \to \R$ has the following formula
\[
H(x,u) = \max_{i,j} \left( u_j- \Upsilon_j(Ax) - u_i \right) \vee 0.
\]
The value functions are typically piecewise quadratic in the explicitly computable examples (see \cite{SS, STA} and the references therein).

\medskip

Another very important direction concerns with the {\it large population double limit},  meaning that the population size $N$ is first taken to infinity, then the noise level $\eta$ is taken to zero. 
As computed in \cite{SS2}, for each noise level $\eta>0$, as $N \to \infty$, the formula of  the corresponding Hamiltonian $H^\eta: X \times Y \to \R$ is
\[
H^\eta(x,u) = \eta \log \left(\sum_{i,j} x_i e^{\frac{u_j-u_i}{\eta}} \frac{e^{\frac{A^j x}{\eta}}}{\sum_k e^{\frac{A^k x}{\eta}}} \right)
\]
We see that $H^\eta(x,0)=0$ for all $x\in X$, and $u \mapsto H^\eta(x,u)$ is convex.
Moreover, $H^\eta \to H$ locally uniformly in $X^\circ \times Y$ as $\eta \to 0$.
It is worth noting however that we do not have that $H^\eta \to H$ locally uniformly on $X \times Y$ as $\eta \to 0$.
In fact, as we will see, $H^\eta$ behaves badly near the boundary of $X$, and in particular, $H^\eta$ is not uniformly coercive on $X$.
See Section \ref{sec:prep}.

Since $H^\eta$ has quite complicated formula and behavior, value functions corresponding to $H^\eta$ and the optimal paths are typically not tractable when $n\geq 3$.
Besides, it was shown in \cite{SS2} that  the Lagrangian $L^\eta(x,\cdot)$, Legendre's transform of $H^\eta(x,\cdot)$, is unbounded and is discontinuous in certain directions as $x\to \partial X$.

\subsection{Some open problems}
Let 
\[
Z= \ol{X \setminus \cB^1} = \cB^2 \cup \cdots \cup \cB^n,
\]
that is, $Z$ is the closure of $X\setminus \cB^1$, be the target set.
We are concerned with the target problem \eqref{eq:TP} with this given $Z$.
A major open question that \cite{SS, SS2} posed is whether the small noise double limit ($\eta \to 0, N \to \infty$ in this order) and the large population double limit ($N \to \infty, \eta \to 0$ in this order) give the same result or not.
In other words, one is concerned whether large deviations properties under the two orders of limits are identical or not.
See the discussion with further details in \cite[Section 8]{SS}.
For some earlier works for the case $n=2$, see \cite{BiSa, S, S12}.
\medskip

In light of Theorems \ref{thm:Source} and \ref{thm:Target}, we are able to phrase this open question in the PDE language as follows.
For each $\eta>0$, let $V^\eta$ be the maximal locally Lipschitz continuous solution to
\begin{equation}\label{eq:SolTarget-eta}
\begin{cases}
H^\eta(x,-DV^\eta(x)) \leq 0\qquad&\text{for almost all }x \in X^\circ;\\
V^\eta(x^\ast) = 0\qquad&\text{for all }x^\ast \in \target.
\end{cases}
\end{equation}
As $H^\eta(x,0)=0$ for all $x \in X$, we see that $V^\eta \geq 0$.

\begin{quest}\label{Q1}
Let $Z= \ol{X \setminus \cB^1}$. 
Let $V$ be the solution to \eqref{eq:TP}.
For each $\eta>0$, let $V^\eta$ be the maximal locally Lipschitz continuous solution to \eqref{eq:SolTarget-eta}.
As $\eta \to 0$, do we have $V^\eta \to V$ uniformly on $X$?
\end{quest}

As noted above, we only have the convergence of $H^\eta$ to $H$ locally uniformly in $X^\circ \times Y$, but not up to the boundary of $X$. 
When $n \geq 3$, the situation becomes much more complicated because of the topology of $X$.
In particular, one can approach $\partial X$ in various different ways, which make the analysis quite hard and delicate.
This leads us to the idea of restricting the convergence problem to compact subsets of $X^\circ$.
For $r>0$ sufficiently small, denote by
\[
y_i = (1-(n-1)r) e_i  + \sum_{ j\neq i } r e_j \qquad \text{ for each } 1\leq i \leq n.
\]
Let $X_r$ be the convex hull of $\{y_1,y_2,\ldots, y_n\}$.
We write $y_i=y_i(r)$ for $1\leq i \leq n$ if needed to demonstrate the clear dependence on $r$.
We now restrict our PDEs to $X_r$ instead of $X$.

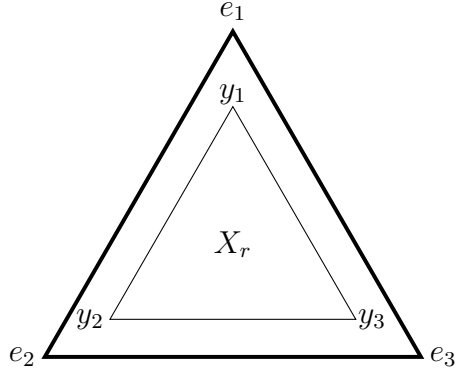
\begin{figure}[h]
\begin{center}
\begin{tikzpicture}
  \draw [line width=1.5pt] (0,0) -- (60:5) -- (5,0) -- cycle;
  \draw  (0.866, 0.5) -- (2.5,3.33) -- (4.134,0.5) -- cycle;
  \draw (2.5, 1.5) node {$X_r$};
 \draw (2.5, 4.6) node {$e_1$};
 \draw (-0.3, 0) node {$e_2$};
 \draw (5.3, 0) node {$e_3$};
 \draw (2.5, 3.5) node {$y_1$};
 \draw (0.6, 0.5) node {$y_2$};
 \draw (4.35, 0.5) node {$y_3$};
 \end{tikzpicture}
\caption{The simplexes $X$ and $X_r$ in case $n=3$}
\end{center}
\end{figure}

Let $V_r$ be the maximal Lipschitz continuous solution to
\begin{equation}\label{eq:SolTarget-r}
\begin{cases}
H(x,-DV_r(x)) \leq 0\qquad&\text{for almost all }x \in X_r^\circ;\\
V_r(x^\ast) = 0\qquad&\text{for all }x^\ast \in \target.
\end{cases}
\end{equation}
For $\eta>0$, let $V^\eta_r$ be the maximal Lipschitz continuous solution to
\begin{equation}\label{eq:SolTarget-eta-r}
\begin{cases}
H^\eta(x,-DV^\eta_r(x)) \leq 0\qquad&\text{for almost all }x \in X_r^\circ;\\
V^\eta_r(x^\ast) = 0\qquad&\text{for all }x^\ast \in \target.
\end{cases}
\end{equation}

\begin{quest}\label{Q2}
Let $Z= \ol{X \setminus \cB^1}$. 
Fix $r>0$ sufficiently small.
Let $V_r$ be the maximal Lipschitz continuous solution to \eqref{eq:SolTarget-r}.
For each $\eta>0$, let $V^\eta_r$ be the maximal Lipschitz continuous solution to \eqref{eq:SolTarget-eta-r}.
As $\eta \to 0$, do we have $V^\eta_r \to V_r$ uniformly on $X_r$?
\end{quest}

\begin{figure}[h]
\begin{center}
\begin{tikzpicture}
\begin{scope}
  \clip  (0,0) -- (60:5) -- (5,0) -- cycle;
  \clip  (-1,3)--(2.3, 1.8)-- (5,2.7)--(5,0)--(0,0)--cycle;
  \fill[gray!20] (-1,3) rectangle (5,0);
\end{scope}
  \draw [line width=0.3pt] (0,0) -- (60:5) -- (5,0) -- cycle;
  \draw [line width=1.5pt] (0.866, 0.5) -- (2.5,3.33) -- (4.134,0.5) -- cycle;
  \draw (2.3, 1.8)-- (2.1,-2);
    \draw (2.3, 1.8)-- (-1,3);
  \draw (2.3, 1.8)-- (5,2.7);
  \draw (2.6, 1.5) node {$X_r$};
    \draw (2.6, 2.6) node {$\cB^1$};
 \draw (2.5, 4.6) node {$e_1$};
 \draw (-0.3, 0) node {$e_2$};
 \draw (5.3, 0) node {$e_3$};
 \end{tikzpicture}
\caption{An example of target set $Z$ (gray region) and $X_r$}
\end{center}
\end{figure}
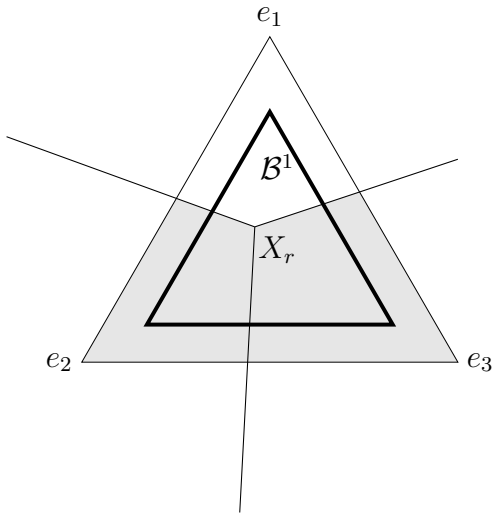
As far as the author is aware of, both Questions \ref{Q1} and \ref{Q2} have not been answered in the literature in general.
In the PDE language,  Questions \ref{Q1} and \ref{Q2} are of selection problem type as the limiting equations \eqref{eq:SolTarget} and \eqref{eq:SolTarget-r}  have many solutions, and it is not clear at all if $\{V^\eta\}$ and $\{V^\eta_r\}$ converge to the corresponding maximal solutions as $\eta \to 0$.

\subsection{Main results}
We first give an affirmative answer to Question \ref{Q1} when $n=2$.
This was already done in \cite{S}, which uses birth-death chain methods to show that in the two-action case, large deviations properties under the two orders of limits are identical.
\begin{thm}\label{thm:logit-2}
Let $n=2$, and $Z= \ol{X \setminus \cB^1} = \cB^2$. 
Let $V$ be the solution to \eqref{eq:TP}.
For each $\eta>0$, let $V^\eta$ be the maximal Lipschitz continuous solution to \eqref{eq:SolTarget-eta}.
Then, $V^\eta \to V$ uniformly on $X$ as $\eta \to 0$.
\end{thm}

It is worth noting that when $n=2$, we are in the one dimensional setting, in which we have explicit formulas for $V^\eta$ and $V$.
We can then utilize these formulas to obtain the convergence result rather straightforwardly.
This is surely not the case for $n\geq 3$.

\medskip

Next, we consider the general case. 
Firstly, we show that \eqref{eq:SolTarget-r} is a good approximation of the target problem \eqref{eq:TP}.

\begin{thm}\label{thm:V-r-to-V}
Let $Z= \ol{X \setminus \cB^1}$, and $V$ be the solution to \eqref{eq:TP}.
For each $r>0$ sufficiently small, let $V_r$ be the maximal Lipschitz continuous solution to \eqref{eq:SolTarget-r}.
As $r \to 0$, $V_r \to V$ locally uniformly in $X^\circ$.
\end{thm}

This shows that theoretically and also practically, it makes sense to consider the problems in a restricted simplex $X_r$ and pass to the limit if necessary.
It is worth noting that \eqref{eq:TP} keeps all of its characteristics and properties in $X_r$.
For state-constraint Hamilton-Jacobi equations in nested domains, see \cite{CDL, AT, KTT, Tu}.

\smallskip

We now give an affirmative answer to Question \ref{Q2} for all $n \geq 2$.

\begin{thm}\label{thm:main2}
Assume $n \geq 2$, and $Z= \ol{X \setminus \cB^1}$.
Fix $r>0$ sufficiently small.
Let $V_r$ be the maximal Lipschitz continuous solution to \eqref{eq:SolTarget-r}.
For each $\eta>0$, let $V^\eta_r$ be the maximal Lipschitz continuous solution to \eqref{eq:SolTarget-eta-r}.
Then $V^\eta_r \to V_r$ uniformly on $X_r$ as $\eta \to 0$.
\end{thm}

\subsection*{Organization of the paper}
The paper is organized as follows.
In Section \ref{sec:2}, we give a proof of Theorem \ref{thm:logit-2}.
We study properties of $H^\eta$, $L^\eta$ and $H$ and give some preparation results for the general case in Section \ref{sec:prep}.
The proof of Theorem \ref{thm:main2} is given in Section \ref{sec:main}.
Some conclusions are discussed in Section \ref{sec:con}.

\subsection*{Acknowledgments}
I thank Bill Sandholm for pointing out to me the problem on interchangeability of double limits in large deviations in games and various enlightening discussions.
I dedicate this paper to the memory of Bill with my greatest respect and admiration.
I thank Srinivas Arigapudi for some useful discussions.

\section{The case of two equilibria}\label{sec:2}

We recall the setting here for clarity.
Let $e_1, e_2$ be the standard basis of $\R^2$. 
Then,
\[
X= \left\{x_1 e_1 + x_2 e_2\colon  x_1, x_2 \geq 0,  \ x_1+ x_2=1 \right\}.
\]
For $x,y \in X$, we write $x\leq y$ if $x_1 \leq y_1$.
Of course, $X$ is $1$-dimensional, and we interpret that
\[
Y=TX= \R^2_0 = \left\{u\in \R^2\colon  u_1+u_2=0\right\}.
\]
We write
\[
A^{1-2} = (\alpha,  -\beta)
\]
for some $\alpha, \beta>0$.
Then 
\[
\mathcal{B}^1 = \{ x \in X \colon  A^{1-2}x \geq 0\} = \left\{x \in  X\colon  x_1 \geq \frac{\beta}{\alpha+\beta}\right\}
\]
and
\[
\mathcal{B}^2 = \{ x \in X \colon  A^{2-1}x \geq 0\} = \left\{x \in  X\colon  x_1 \leq \frac{\beta}{\alpha+\beta}\right\}.
\]
The Hamiltonian $H: X \times Y \to \R$ has the following formula
\[
H(x,u) = \max_{i,j} \left( u_j- \Upsilon_j(Ax) - u_i \right) \vee 0.
\]
For $\eta>0$,  the formula of $H^\eta$ is
\[
H^\eta(x,u) = \eta \log \left(\sum_{i,j} x_i e^{\frac{u_j-u_i}{\eta}} \frac{e^{\frac{A^j x}{\eta}}}{\sum_k e^{\frac{A^k x}{\eta}}} \right)
\]

\begin{proof}[Proof of Theorem \ref{thm:logit-2}]
We need to understand about the zero level set of $H$ and zero sublevel set of $H^\eta$,
which give us formulas for $V$ and $V^\eta$.

\medskip

Let $p=u_1-u_2$.
For $x\in \mathcal{B}^1$, 
\begin{align*}
H(x,-u)&= \max \{0, u_2-u_1, u_1-u_2-A^{1-2}x\}\\
&=\max\{0, -p, p - A^{1-2}x\},
\end{align*}
which gives us the explicit form of zero set 
\begin{equation}\label{N-x}
N(x) =\{p \in \R\colon H(x,-u)=0\}= \left[0, {A^{1-2}x}\right].
\end{equation}
By using this and the maximality of $V$, we get that $V(x) = 0$ for $x\in Z$, and, for $x\in \cB^1$,
\begin{equation}\label{V-1d}
V(x)=\int_{\frac{\beta}{\al+\beta}}^{x_1} A^{1-2} (s e_1 + (1-s)e_2)\,ds.
\end{equation}

Next, for $\eta>0$ and $x\in \mathcal{B}^1$, denote by $N^\eta(x)=\{p \in \R\colon  H^\eta(x,-u)  \leq 0\}$.
It is clear that $H^\eta(x,-u) \leq 0$ if and only if
\begin{align*}
&\sum_{i,j} x_i e^{\frac{u_i-u_j}{\eta}} e^{\frac{A^j x}{\eta}} \leq \sum_k e^{\frac{A^k x}{\eta}}\\
\iff \  & x_1 \left( e^{\frac{A^1 x}{\eta}} + e^{\frac{p}{\eta}} e^{\frac{A^2 x}{\eta}} \right)
+x_2 \left( e^{\frac{A^2 x}{\eta}} + e^{\frac{-p}{\eta}} e^{\frac{A^1 x}{\eta}} \right) \leq e^{\frac{A^1 x}{\eta}} + e^{\frac{A^2 x}{\eta}}\\
\iff \  & x_1 e^{\frac{p}{\eta}} e^{\frac{A^2 x}{\eta}} + x_2  e^{\frac{-p}{\eta}} e^{\frac{A^1 x}{\eta}}
\leq x_2 e^{\frac{A^1 x}{\eta}} +x_1 e^{\frac{A^2 x}{\eta}}.
\end{align*}
One point to notice is that the product of two terms on left hand side is equal to that of two terms on the right hand side.
It is clear that the above inequality becomes equality if $p=0$.
The other case that equality happens is when
\begin{align*}
&x_1 e^{\frac{p}{\eta}} e^{\frac{A^2 x}{\eta}}  = x_2 e^{\frac{A^1 x}{\eta}}\\
\iff \ & p =  A^{1-2}x +  \eta  (\log x_2 - \log x_1).
\end{align*}
For each $\eta>0$ small, there exists $\delta_\eta \in (0,1)$ such that
\begin{align*}
&A^{1-2}(\delta_\eta e_1 + (1-\delta_\eta)e_2) +  \eta  (\log (1-\delta_\eta) - \log \delta_\eta)=0\\
\iff \ & (\al+\beta) \delta_\eta +\eta \log \left(\frac{1-\delta_\eta}{\delta_\eta}\right) = \beta,
\end{align*}
and $\lim_{\eta \to 0} \delta_\eta=1$.
Therefore, for $x\in  \mathcal{B}^1$ with $x_1 \leq \delta_\eta$,
\begin{equation*}\label{eq:N-eta}
N^\eta(x) = \left[0, A^{1-2}x +  \eta  (\log x_2 - \log x_1) \right].
\end{equation*}
And, for $x\in  \mathcal{B}^1$ with $x_1 \geq \delta_\eta$,
\begin{equation*}
0 \in N^\eta(x) \subset (-\infty,0].
\end{equation*}
Based on these, we have the following explicit formula for $V^\eta$.
Surely, $V^\eta=0$ on $Z$.
For $x \in \cB^1$ such that $\frac{\beta}{\al+\beta} \leq x_1 \leq \delta_\eta$,
\[
V^\eta(x)= \int_{\frac{\beta}{\al+\beta}}^{x_1} \left(A^{1-2} (s e_1 + (1-s)e_2) +  \eta  (\log (1-s) - \log s) \right)\,ds.
\]
And, $V^\eta(x) = V^\eta(  \delta_\eta e_1 + (1- \delta_\eta) e_2))$ for $ \delta_\eta \leq x_1 \leq 1$.

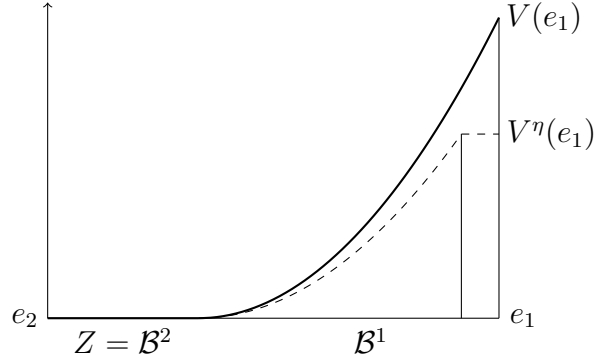
\begin{figure}[h]
\begin{center}
\begin{tikzpicture}
      \draw (0,0) -- (6,0) node[right] {$e_1$};
      \draw[->] (0,0) -- (0,4.2);
      \draw[thick, scale=1,domain=2:6,smooth,variable=\x] plot ({\x},{(\x-2)*(\x-2)/4});
      \draw[thick] (0,0)--(2,0);
      \draw[dashed, scale=1,domain=2:5.5,smooth,variable=\x] plot ({\x},{(\x-2)*(\x-2)/5});
      \draw[dashed] (5.5, 2.45)--(6, 2.45);
      \draw (5.5, 2.45)--(5.5, 0);
      \draw (6,4)--(6, 0);
      \draw (-0.3,0) node {$e_2$};
      \draw (1,-0.3) node {$Z=\cB^2$};
      \draw (4.3,-0.3) node {$\cB^1$};
      \draw (6.7,2.45) node {$V^\eta(e_1)$};
      \draw (6.6,4) node {$V(e_1)$};
 \end{tikzpicture}
\caption{Graph of $V^\eta$ (dashed curve) and $V$ (solid curve)}
\end{center}
\end{figure}

Thus, $V^\eta=V=0$ on $Z$.
For $x \in \cB^1$ such that $\frac{\beta}{\al+\beta} \leq x_1 \leq \delta_\eta$,
\begin{align*}
|V^\eta(x) - V(x)| &= \left| \int_{\frac{\beta}{\al+\beta}}^{x_1}  \eta  (\log (1-s) - \log s) \,ds \right|\\
&= \eta  \left| (-s\log s - (1-s)\log(1-s))\Big|_{s=\frac{\beta}{\al+\beta}}^{x_1} \right| \leq C\eta.
\end{align*}
For $x \in \cB^1$ such that $ \delta_\eta \leq x_1 \leq 1$,
\begin{align*}
&|V^\eta(x) - V(x)| \\
=\,  & \left| V^\eta(  \delta_\eta e_1 + (1- \delta_\eta) e_2)) - V(  \delta_\eta e_1 + (1- \delta_\eta) e_2)) + V(  \delta_\eta e_1 + (1- \delta_\eta) e_2))- V(x) \right|\\
\leq\, & C\eta + C(1-\delta_\eta).
\end{align*}
Hence, we conclude that
\[
\|V^\eta - V\|_\infty  \leq  C\eta + C(1-\delta_\eta) \to 0 \qquad \text{ as } \eta \to 0.
\]
\end{proof}

\begin{rem}
In the above proof, we see that, for $ \delta_\eta \leq x_1 \leq 1$,
\begin{equation*}
\begin{cases}
N^\eta(x) = \left[A^{1-2}x +  \eta  (\log x_2 - \log x_1),0 \right],\\
\lim_{x_1 \to 1^-} \left(A^{1-2}x +  \eta  (\log x_2 - \log x_1)\right) = -\infty.
\end{cases}
\end{equation*}
This shows that $H^\eta$ behaves in a quite singular way as $x_1 \to 1^-$.
In particular, when $x_1=1$, that is, $x=e_1$, we have
\[
H^\eta(e_1,-u)=H^\eta(e_1,p)= \eta \log \left( \frac{e^{\frac{A^1 e_1}{\eta}}+e^{\frac{p}{\eta}}e^{\frac{A^2 e_1}{\eta}}}{e^{\frac{A^1 e_1}{\eta}}+e^{\frac{A^2 e_1}{\eta}}}\right).
\]
It is clear that $H^\eta(e_1,p)>0$ for $p>0$, $H^\eta(e_1,0)=0$, and $H^\eta(e_1,p)<0$ for $p<0$.
Moreover,
\[
\lim_{p \to -\infty} H^\eta(e_1,p)= \eta \log \left( \frac{e^{\frac{A^1 e_1}{\eta}}}{e^{\frac{A^1 e_1}{\eta}}+e^{\frac{A^2 e_1}{\eta}}}\right) \geq \eta \log\left(\frac{1}{2}\right)=-\eta \log 2.
\]
Thus, $H^\eta$ is not coercive on $X$.
A careful computation gives us that
\[
\lim_{\eta \to 0} H^\eta(e_1,-u) = \lim_{\eta \to 0} H^\eta(e_1,p)= \max\{0, p - A^{1-2} e_1\} \neq H(e_1,p).
\]

We only have that, for each  compact subset $X_1$ of $X^\circ$,  $H^\eta$ is uniformly coercive on $X_1$, that is,
\[
\lim_{|p| \to \infty} \inf_{\eta \in (0,1)} \min_{x \in X_1} H^\eta(x,p) = +\infty;
\]
and $H^\eta \to H$ uniformly on $X_1 \times B_R$ as $\eta \to 0$ for each given $R>0$.
\end{rem}

\begin{rem}
It is clear from the above proof that it also gives a proof to Theorem \ref{thm:main2} in case $n=2$.
\end{rem}

\section{The general case -- Preparation steps} \label{sec:prep}
In this section, assumptions of Theorems \ref{thm:V-r-to-V} and \ref{thm:main2} are always in charge.

\subsection{Some analysis on properties of $H^\eta$, $L^\eta$, and $H$}
It is clear that in order to answer Questions \ref{Q1} and \ref{Q2}, we need to have a deeper understanding of $H^\eta$, $L^\eta$, and $H$ in this logit choice protocol.

Recall that $X$ is the convex hull of $\{e_1, e_2, \ldots, e_n\}$, the standard basis of $\R^n$.
For each $r>0$ sufficiently small, $X_r$ is the convex hull of $\{y_1,y_2,\ldots, y_n\}$.
Besides,
\[
Y=\left\{u\in \R^n\colon  \sum_{i=1}^n u_i=0\right\} = \left\{ p\in \R^{n-1} \colon p=(u_2-u_1,\ldots, u_n-u_1) \right\}.
\]
Although it is not standard, we write $p=(p_2,\ldots, p_n)$ to make things consistent in terms of notions, that is, $p_i = u_i-u_1$ for $2\leq i \leq n$.
By abuse of notions, we write $H^\eta(x,u)=H^\eta(x,p)$, $H(x,u)=H(x,p)$.

\begin{lem}
For each $r>0$ sufficiently small,  $H^\eta$ is uniformly coercive on $X_r$, that is,
\begin{equation}\label{H-eta-uniformly-coercive}
\lim_{|p| \to \infty} \inf_{\eta \in (0,1)} \min_{x \in X_r} H^\eta(x,p) = +\infty.
\end{equation}

\end{lem}

\begin{proof}
Recall that
\[
H^\eta(x,u) =H^\eta(x,p)= \eta \log \left(\sum_{i,j} x_i e^{\frac{u_j-u_i}{\eta}} \frac{e^{\frac{A^j x}{\eta}}}{\sum_k e^{\frac{A^k x}{\eta}}} \right).
\]
We only need to consider the case where $x \in \cB^1 \cap X_r$ as other cases can be done analogously.
For each $2 \leq j \leq n$, it is clear that
\[
H^\eta(x,p) \geq \eta  \log \left(x_1 e^{\frac{u_j-u_1}{\eta}} \frac{e^{\frac{A^j x}{\eta}}}{n e^{\frac{A^1 x}{\eta}}} \right)=p_j - A^{1-j}x + \eta(\log x_1 - \log n);
\]
and
\[
H^\eta(x,p) \geq \eta  \log \left(x_j e^{\frac{u_1-u_j}{\eta}} \frac{e^{\frac{A^1 x}{\eta}}}{n e^{\frac{A^1 x}{\eta}}} \right)=-p_j + \eta(\log x_j - \log n).
\]
Therefore, we arrive at
\begin{align*}
H^\eta(x,p) &\geq \max_{2 \leq j \leq n} \left\{p_j - A^{1-j}x + \eta(\log r - \log n), -p_j + \eta(\log r - \log n) \right\}\\
&\geq \frac{1}{n-1} |p| - \max_{2\leq j \leq n} A^{1-j}x + \eta(\log r - \log n),
\end{align*}
which gives \eqref{H-eta-uniformly-coercive}.
\end{proof}

In the above proof, for $x\in \cB^1 \cap X_r$, we really need to use the property that $x_j \geq r$ to have that $\log x_j \geq \log r$, which is important in our uniform coercivity claim.
For Hamilton-Jacobi equations, uniform coercivity yields Lipschitz estimates, which are essential in our analysis.

\medskip

As $x \to e_1$, we have that $x_j \to 0$ for $2\leq j \leq n$, and $\log x_j \to -\infty$.
As a result, we lose uniform coercivity of $H^\eta$ in $X^\circ$ as follows.

\begin{lem} \label{lem:H-eta-e1}
$H^\eta$ is not coercive at $x=e_1$, and is not uniformly coercive in $X^\circ$.
\end{lem}

\begin{proof}
We first show that $H^\eta$ is not coercive at $x=e_1$.
Fix $x=e_1$, and $\bar u = (0,-1,-1,\ldots,-1)$, which means that $p=(-1,-1,\ldots,-1)$.
For $s>0$, we see that
\begin{align*}
H^\eta(e_1, s\bar u ) =H^\eta(e_1,sp)= \eta \log \left (\frac{e^{\frac{A^1 e_1}{\eta}} + \sum_{j>1}e^{\frac{A^j e_1}{\eta}} e^{\frac{-s}{\eta}}}{\sum_k e^{\frac{A^k e_1}{\eta}}}  \right) < \eta \log 1 =0,
\end{align*}
and
\begin{align*}
H^\eta(e_1, s\bar u ) = \eta \log \left (\frac{e^{\frac{A^1 e_1}{\eta}} + \sum_{j>1}e^{\frac{A^j e_1}{\eta}} e^{\frac{-s}{\eta}}}{\sum_k e^{\frac{A^k e_1}{\eta}}}  \right) \geq \eta \log \left (\frac{e^{\frac{A^1 e_1}{\eta}} } {\sum_k e^{\frac{A^k e_1}{\eta}}}  \right) > - \eta \log n.
\end{align*}
Thus, for $s>0$,
\[
 - \eta \log n < H^\eta(e_1, sp )  < 0,
\]
and
\[
 \lim_{s \to \infty} H^\eta(e_1, sp ) =\eta \log \left (\frac{e^{\frac{A^1 e_1}{\eta}} } {\sum_k e^{\frac{A^k e_1}{\eta}}}  \right) \in ( - \eta \log n,0).
\]

Let us show that $H^\eta$ is not uniformly coercive as $x \to \partial X$, which means that the maximal solution $V^\eta$ has complicated behavior near the boundary.
In a similar fashion, for $x \in X$, we can estimate that
\begin{align*}
H^\eta(x, sp) &\leq \eta \log \left(x_1 + \sum_{j>1} x_j e^{\frac{s}{\eta}}\right) \\
&= \eta \log( 1 + (1-x_1)(e^{\frac{s}{\eta}}-1)) \leq (1-x_1)\eta (e^{\frac{s}{\eta}}-1).
\end{align*}
We use the fact that $\log(1+r) \leq r$ for $r\geq 0$ in the last inequality above.
So, clearly, as $x\to e_1$, we have $1-x_1 \to 0$, and thus,
\[
\limsup_{x \to e_1} H^\eta(x, sp) \leq 0.
\]

\end{proof}

\begin{prop}
For each $r>0$ sufficiently small,  $H^\eta \to H$ uniformly  on $X_r \times Y$ as $\eta \to 0$.
However, $H^\eta$ does not converge to $H$ locally uniformly on $X \times Y$ as $\eta \to 0$.
\end{prop}

\begin{proof}
Fix $r>0$ sufficiently small.
We only need to consider the case where $x \in \cB^1 \cap X_r$ as other cases can be done analogously.
For each $i,j \in \{1,2,\ldots,n\}$, it is clear that
\begin{align*}
H^\eta(x,u) &= \eta \log \left(\sum_{i,j} x_i e^{\frac{u_j-u_i}{\eta}} \frac{e^{\frac{A^j x}{\eta}}}{\sum_k e^{\frac{A^k x}{\eta}}} \right) \\
& \geq \eta  \log \left(x_i e^{\frac{u_j-u_i}{\eta}} \frac{e^{\frac{A^j x}{\eta}}}{n e^{\frac{A^1 x}{\eta}}} \right)=u_j-u_i - A^{1-j}x + \eta(\log x_i - \log n)\\
&\geq u_j-u_i - A^{1-j}x + \eta(\log r - \log n).
\end{align*}
In particular, for $i=j=1$, we see that $H^\eta(x,u) \geq \eta(\log r - \log n)$.
Thus, take maximum over $i,j \in \{1,2,\ldots,n\}$ in the above inequalities to deduce that
\begin{equation}\label{bound-1}
H^\eta(x,u) \geq H(x,u) + \eta(\log r - \log n).
\end{equation}
To get the other bound, we assume that $H(x,u) = u_k - u_l - A^{1-k} x$ for some $k,l \in \{1,2,\ldots,n\}$.
This means that, for all $i,j \in \{1,2,\ldots,n\}$,
\[
u_k - u_l + A^k x \geq u_j - u_i + A^j x.
\]
Therefore,
\begin{align*}
H^\eta(x,u) &= \eta \log \left(\sum_{i,j} x_i e^{\frac{u_j-u_i}{\eta}} \frac{e^{\frac{A^j x}{\eta}}}{\sum_k e^{\frac{A^k x}{\eta}}} \right) \\
& \leq \eta  \log \left( \sum_{i,j} x_i e^{\frac{u_k-u_l}{\eta}} \frac{e^{\frac{A^k x}{\eta}}}{e^{\frac{A^1 x}{\eta}}} \right)=u_k-u_l - A^{1-k}x + \eta  \log n\\
&= H(x,u)+ \eta \log n.
\end{align*}
Combine this with \eqref{bound-1}, we arrive at
\begin{equation}\label{bound-2}
\|H^\eta - H\|_{L^\infty(X_r \times Y)}  \leq \eta (|\log r| + \log n)= \eta ( \log n - \log r).
\end{equation}
Therefore, $H^\eta \to H$ uniformly  on $X_r \times Y$ as $\eta \to 0$.
However, $H^\eta$ does not converge to $H$ locally uniformly on $X \times Y$ as $\eta \to 0$ thanks to Lemma \ref{lem:H-eta-e1}.
\end{proof}

\smallskip

We also include here a different proof showing that $L^\eta(x,\cdot)$ blows up as $x\to \partial X$ in certain directions (see inequality (31) in \cite{SS2}).

\begin{lem}\label{lem:L-eta}
For $x\in X^\circ$ and $i \neq j$,
\[
L^\eta(x,e_j -  e_i) \geq -\eta \log x_i - \eta \log 2.
\]
In particular, $L^\eta(x,e_j -e_i) \to +\infty$ as $x_i \to 0$.
\end{lem}

\begin{proof}
Without loss of generality, we consider $i=2, j=1$.
By the Legendre transform,
\begin{align*}
L^\eta(x,e_1-e_2) &= \sup_{u \in Y} ( (e_1- e_2)\cdot u - H^\eta(x,u))\\
&\geq  \sup_{u \in Y} \left( (u_1- u_2)  - \eta \log \left( \sum_{i,j}  x_i e^{\frac{u_j-u_i}{\eta}} \right) \right)
\end{align*}
Pick $u_1=k$, $u_2=-k$, and $u_j=0$ for $j >2$.
For $k >0$ sufficiently large depending on $x$, it is clear that
\[
\sum_{i,j}  x_i e^{\frac{u_j-u_i}{\eta}} \leq x_2 e^{\frac{2k}{\eta}} + \sum_{i,j}  x_i e^{\frac{k}{\eta}}=x_2 e^{\frac{2k}{\eta}} +n  e^{\frac{k}{\eta}} <2 x_2 e^{\frac{2k}{\eta}}.
\]
Thus, for $k >0$ sufficiently large,
\[
L^\eta(x,e_1-e_2) \geq 2k - \eta \log \left(2 x_2 e^{\frac{2k}{\eta}} \right) = -\eta \log(2x_2) = -\eta \log x_2 - \eta \log 2.
\]
\end{proof}
Denote by $c^\eta(x,y)$ the minimum cost of traveling from $x$ to $y$ using the Lagrangian $L^\eta$, that is,
\begin{multline*}
c^\eta(x,y) = \inf\Big\{ \int_0^T L^\eta(\phi(s),\dot \phi(s))\,ds \colon \phi \in \Phi_T\text{ for some }T \geq 0, \\
 \phi(0) = x, \phi(T) =y  \Big\}.
\end{multline*}

\begin{lem}
For $0<\eta<1$ and $0<r<\frac{1}{2(n-1)}$ sufficiently small, there exists a constant $C>0$ depending only on $n$ and $A$ such that
\[
c^\eta(e_1,y_1) = c^\eta(e_1, y_1(r)) \leq Cr.
\]
\end{lem}

\begin{proof}
For each $x\in \cB^1$ and $2\leq j \leq n$, by the Legendre transform,
\begin{align*}
L^\eta(x,e_j-e_1) &= \sup_{u \in Y} ( (e_j- e_1)\cdot u - H^\eta(x,u))\\
&\leq  \sup_{u \in Y} \left( (u_j- u_1)  - \eta \log \left( x_1 e^{\frac{u_j-u_1}{\eta}} \frac{e^{\frac{A^j x}{\eta}}}{n e^{\frac{A^1 x}{\eta}}} \right ) \right)\\
&=A^{1-j}x + \eta (\log n- \log x_1).
\end{align*}
Let $\bar v = (e_2+\cdots + e_n)/(n-1) - e_1$.
By convexity of $L^\eta(x,\cdot)$,
\begin{align*}
L^\eta(x,\bar v) &\leq \frac{1}{n-1} \sum_{j=2}^n L^\eta(x,e_j-e_1) \leq \frac{1}{n-1} \sum_{j=2}^n A^{1-j}x + \eta (\log n- \log x_1)\\
&\leq A^1 x+ \eta (\log n- \log x_1).
\end{align*}
Denote by
\[
\gam(s) = e_1 + s \bar v \qquad \text{ for } s \in [0, (n-1)r].
\]
Then, $\gam(0) = e_1$, $\gam((n-1)r)=y_1=y_1(r)$, and
\begin{align*}
c^\eta(e_1, y_1(r)) &\leq \int_0^{(n-1)r} L^\eta(\gam(s),\dot \gam(s))\,ds \\
&\leq (n-1)r \left(\max_{x\in \cB^1} A^1x \right) + \eta r (n-1)\left(\log n - \log (1-(n-1)r) \right)\\
&\leq C(1+\eta) r \leq Cr
\end{align*}
\end{proof}
Thus, $c^\eta(e_1, y_1(r)) \leq Cr$, which means that the cost of transitioning from $e_1$ to $y_1(r)$ is small enough of order $O(r)$, and is vanishing as $r \to 0$.
This gives another evident that it makes sense to consider the problem in the restricted simplex $X_r$ in place of $X$.

\medskip

We now study about the zero level set of $H$, which is essential in our analysis later.
\begin{lem}\label{lem:N}
Fix $x\in \mathcal{B}^1$, and denote by 
\[ 
N(x)=\{u \in Y \colon  H(x,u)=0\} =\{p\in \R^{n-1} \colon H(x,p)=0\}.
\]
Then,
\[
N(x) = [0,A^{1-2}x]\times\cdots \times [0,A^{1-n} x].
\]
\end{lem}

\begin{proof}
For $x\in \mathcal{B}^1 \cap X_r$, we have
\begin{align*}
H(x,u) &= \max_{i,j} \left( u_j- \Upsilon_j(Ax) - u_i \right) \vee 0\\
&= \max_{i,j} \left( u_j- A^{1-j}x - u_i \right) \vee 0.
\end{align*}
For $i=1$ and $j>1$, we see that 
\[
u_j- A^{1-j}x - u_1 = p_j - A^{1-j}x.
\]
For $j=1$ and $i>1$, we have that
\[
u_1- A^{1-1}x - u_i = -p_i.
\]
For $i,j >1$,
\[
u_j- A^{1-j}x - u_i  = (u_j-u_1 - A^{1-j}x) - (u_i - u_1 ) = p_j - A^{1-j} x - p_i.
\]
We use the three identities above to deduce that
\[
N(x)= \left\{ p \in \R^{n-1} \colon  p_2 \in [0,A^{1-2}x], \ldots, p_n \in [0,A^{1-n}x] \right\}.
\]

\end{proof}

\subsection{Preparation results}

\begin{prop}\label{prop:subseq-conv}
Fix $r>0$ sufficiently small.
There exists a sequence $\{\eta_k\} \to 0$ such that $V^{\eta_k}_r \to \bar V$ locally uniformly in $X^\circ$, where $\bar V$ solves \eqref{eq:SolTarget-r}.
\end{prop}

\begin{proof}
Thanks to \eqref{H-eta-uniformly-coercive},  $H^\eta$ is uniformly coercive on $X_r$, that is,
\[
\lim_{|p| \to \infty} \inf_{\eta \in (0,1)} \min_{x \in X_r} H^\eta(x,p) = +\infty;
\]
and $H^\eta \to H$ uniformly on $X_r \times B_R$ as $\eta \to 0$ for each given $R>0$.
Therefore, there exists $C=C(r)>0$ such that 
\begin{equation}\label{eq:bound-DV-eta}
\|DV^\eta_r\|_{L^\infty(X_r)} \leq C(r).
\end{equation}
Since $V^\eta_r=0$ on $\target$ and $X_r$ is compact, we imply that, there exists a constant $C=C(r)>0$ independent of $\eta$ such that
\begin{equation}\label{eq:bound-V-eta}
\|V^\eta_r\|_{L^\infty(X_r)}+\|DV^\eta_r\|_{L^\infty(X_r)} \leq C(r).
\end{equation}

Thanks  to \eqref{eq:bound-V-eta}, we use the Arzel\`a-Ascoli theorem  to find a sequence $\{\eta_k\} \to 0$ such that $V^{\eta_k}_r \to \bar V$ uniformly on $X_r$.
It is clear that $\bar V$ satisfies the bound \eqref{eq:bound-V-eta} as well.

We now need to show that $\bar V$ is a solution to \eqref{eq:SolTarget-r}.
Surely $\bar V=0$ on $Z$.
The subsolution test follows the classical argument on stability of viscosity subsolutions. 
Let us present it here anyway for the sake of completeness.

Take a smooth test function $\varphi$ such that $\bar V - \varphi$ has a strict maximum at $y \in X^\circ_r \setminus \target$.
Take $s>0$ such that $B_s(y) \subset X^\circ_r \setminus \target$.
As $V^{\eta_k}_r \to V$ uniformly on $X_r$, for $k$ sufficiently large, we have that $V^{\eta_k}_r - \varphi$ has a local maximum at $y_k \in B_s(y)$, and $\lim_{k \to \infty} y_k=y$.
Since $V^{\eta_k}_r - \varphi$ has a local maximum at $y_k \in B_s(y)$, by the definition of viscosity subsolutions to \eqref{eq:SolTarget-eta}, 
\[
H^{\eta_k}(y_k, -D\varphi(y_k)) \leq 0.
\]
As $\varphi$ is smooth and $H^{\eta_k} \to H$ locally uniformly on $X_r \times Y$, we let $k \to \infty$ in the above to get
\[
H(y,-D\varphi(y))= \lim_{k \to \infty} H^{\eta_k}(y_k, -D\varphi(y_k)) \leq 0.
\]
\end{proof}

Proposition \ref{prop:subseq-conv} is an important step towards answering a main question in our paper, Question \ref{Q2}. 
Nevertheless, it is not enough to conclude here as we do not know yet whether $\bar V$ is the maximal solution to \eqref{eq:SolTarget-r} or not. As it was shown in \cite{STA}, \eqref{eq:SolTarget-r} has infinitely many solutions, and any convex combination of two solutions is again a solution, which means that we need deeper understanding of the situations to get the maximality.
In the next section, we need to do a delicate and deeper analysis to yield the maximality property.

\medskip

We next show that $V_r$ to $V$ locally uniformly in $X^\circ$ as $r\to 0$.

\begin{proof}[Proof of Theorem \ref{thm:V-r-to-V}]
By repeating the same proof to that of Proposition \ref{prop:subseq-conv}, there exists a subsequence $\{r_k\} \to 0$ such that $V_{r_k} \to \bar V$ locally uniformly in $X^\circ$,
$\bar V$ solves \eqref{eq:SolTarget} and 
\begin{equation}\label{VrV-1}
\bar V \leq V.
\end{equation}
We just need to obtain the reverse inequality. 
This is, in fact, not so hard to see.
For each $k \in \N$, $V$ is a solution to \eqref{eq:SolTarget-r} with $r=r_k$.
Hence,
\[
V \leq V_{r_k} \qquad \text{ on } X_{r_k}.
\]
Let $k \to \infty$ in the above to imply that, for each $r>0$,
\begin{equation}\label{VrV-2}
V \leq \bar V \qquad \text{ on } X_r.
\end{equation}
Combine \eqref{VrV-1} and \eqref{VrV-2} to conclude.

\end{proof}

In fact, in the above proof, we have shown that $V_r \to V$ locally uniformly in $X^\circ$ in a decreasing way.

\section{Proof of Theorem \ref{thm:main2}} \label{sec:main}
\subsection{The case of three equilibria}

The situation in this case is quite complicated because of the two dimensional topology of $X$.
We choose to do the case $n=3$ first to show our ideas clearly.
The general case is done in a similar manner afterwards.

\begin{proof}[Proof of Theorem \ref{thm:main2} in case $n=3$]
Again, we need to investigate the zero level set of $H$ and zero sublevel set of $H^\eta$ first.
By Lemma \ref{lem:N}, for $x\in \cB^1$, we have that
\[ 
N(x)=\{p\in \R^2 \colon H(x,p)=0\} = [0,A^{1-2}x]\times [0,A^{1-3} x].
\]
First, fix $\delta>0$ small enough and $\theta \in (1/2,1)$ such that $\theta$ is close to $1$. Denote by
\[
D=\left\{ x\in \mathcal{B}^1 \cap X_r\colon  V_r(x) \leq \delta\right\}.
\]

For $x\in (\mathcal{B}^1 \cap X_r) \setminus D$, we see that there exists $\bar \delta>0$ such that  
\[
\min\{A^{1-2}x, A^{1-3}x\} \geq \bar \delta.
\]
 Let $N^\eta(x)=\{p \in \R^2 \colon  H^\eta(x,p) \leq 0\}$.
We compute that $H^\eta(x,p) \leq 0$ if and only if
\begin{align*}
&\sum_{i,j} x_i e^{\frac{u_j-u_i}{\eta}} e^{\frac{A^j x}{\eta}} \leq \sum_k e^{\frac{A^k x}{\eta}}\\
\iff \  & \left(x_1 e^{\frac{u_2-u_1}{\eta}} e^{\frac{A^2x}{\eta}} + x_2 e^{\frac{u_1-u_2}{\eta}} e^{\frac{A^1x}{\eta}} \right)
+\left(x_1 e^{\frac{u_3-u_1}{\eta}} e^{\frac{A^3x}{\eta}} + x_3 e^{\frac{u_1-u_3}{\eta}} e^{\frac{A^1x}{\eta}} \right)\\
&\quad +\left(x_2 e^{\frac{u_3-u_2}{\eta}} e^{\frac{A^3x}{\eta}} + x_3 e^{\frac{u_2-u_3}{\eta}} e^{\frac{A^2x}{\eta}} \right)\\
&\leq \left( x_2 e^{\frac{A^1x}{\eta}} + x_1 e^{\frac{A^2x}{\eta}}\right)
+ \left( x_1 e^{\frac{A^3x}{\eta}} + x_3 e^{\frac{A^1x}{\eta}}\right)
+ \left( x_2 e^{\frac{A^3x}{\eta}} + x_3 e^{\frac{A^2x}{\eta}}\right).
\end{align*}
It is useful to compare each pair in brackets on left hand side with its corresponding pair on right hand side.
We find particular points that are in $N^\eta(x)$.
First, $0\in N^\eta(x)$.
Second, $p\in N^\eta(x)$ where
\[
\begin{cases}
p_2=u_2-u_1=A^{1-2}x+\eta(\log x_2 - \log x_1),\\
p_3=u_3-u_1=A^{1-3}x+\eta(\log x_3 - \log x_1).
\end{cases}
\]

Third, for $q\in \R^2$ with $q_2=u_2-u_1=0$, $q_3=u_3-u_1=\theta A^{1-3}x$, we see that the first brackets on both sides are the same, and
\begin{align*}
&\left(x_1 e^{\frac{u_3-u_1}{\eta}} e^{\frac{A^3x}{\eta}} + x_3 e^{\frac{u_1-u_3}{\eta}} e^{\frac{A^1x}{\eta}} \right)
+\left(x_2 e^{\frac{u_3-u_2}{\eta}} e^{\frac{A^3x}{\eta}} + x_3 e^{\frac{u_2-u_3}{\eta}} e^{\frac{A^2x}{\eta}} \right)\\
\leq \ & (x_1+x_2+2x_3) e^{\frac{A^3x+\theta A^{1-3}x}{\eta}}
\leq 2 e^{\frac{A^3x+\theta A^{1-3}x}{\eta}} \leq x_3 e^{\frac{A^1x}{\eta}} \\
\leq \ &  \left( x_1 e^{\frac{A^3x}{\eta}} + x_3 e^{\frac{A^1x}{\eta}}\right)
+ \left( x_2 e^{\frac{A^3x}{\eta}} + x_3 e^{\frac{A^2x}{\eta}}\right),
\end{align*}
provided that $\eta (\log 2 - \log x_3) \leq (1-\theta)\bar \delta$, which is used in the last inequality of line $2$ in the computation right above.
So $q\in N^\eta(x)$ for $\eta>0$ sufficiently small such that $\eta (\log 2 - \log r) \leq (1-\theta)\bar \delta$.
Similarly, for $w \in \R^2$ with $w_3=u_3-u_1=0$, $w_2=u_2-u_1=\theta A^{1-2}x$, we have $w\in N^\eta(x)$.

Therefore, for $\eta>0$ such that $\eta(\log 2 - \log r)\leq (1-\theta)\bar \delta$, we get that the convex hull of $\{0,p,q,w\}$ is a subset of $N^\eta(x)$. Thus, for $\eta$ small enough,
\begin{equation}\label{eq:scale-theta}
\theta N(x) \subset N^\eta(x).
\end{equation}
Define
\[
\varphi^\theta(x)=
\begin{cases}
0 \quad &\text{ for } x \in \target \cup D,\\
\theta(V_r(x)-\delta) \quad &\text{ for } x\in (\mathcal{B}^1 \cap X_r)\setminus D.
\end{cases}
\]
Then $D\varphi^\theta(x)=0$ for $x \in (\target \cup D)^\circ$, and if $V_r$ is differentiable at $x \in (\mathcal{B}^1 \cap X_r)\setminus D$ then
\[
D\varphi^\theta(x) = \theta DV_r(x) \in \theta N(x) \subset N^\eta(x).
\]
Thus, for $\eta>0$ small enough, $\varphi^\theta$ is a solution to \eqref{eq:SolTarget-eta-r}.
We yield $V^\eta_r \geq \varphi^\theta$.
Combine this with Proposition \ref{prop:subseq-conv} to imply, for any sequence $\{\eta_k\} \to 0$ such that $V^{\eta_k}_r \to \bar V$ uniformly on $X_r$,
\[
\bar V \geq \varphi^\theta.
\]
We get the desired result by letting $\theta \to 1$ and $\delta \to 0$ in this order.
\end{proof}

\subsection{The general case}

\begin{proof}[Proof of Theorem \ref{thm:main2} in case $n \geq 3$]
By Lemma \ref{lem:N}, for $x\in \cB^1$, we have that
\[ 
N(x)=\{p\in \R^{n-1} \colon H(x,p)=0\} = [0,A^{1-2}x]\times [0,A^{1-3} x] \times \cdots \times [0,A^{1-n}x].
\]
First, fix $\delta>0$ small enough and $\theta \in (1/2,1)$ such that $\theta$ is close to $1$. Denote by
\[
D=\left\{ x\in \mathcal{B}^1 \cap X_r\colon  V_r(x) \leq \delta\right\}.
\]

For $x\in (\mathcal{B}^1 \cap X_r) \setminus D$, we see that there exists $\bar \delta>0$ such that  
\[
\min\{A^{1-2}x, A^{1-3}x, \ldots ,A^{1-n}x\} \geq \bar \delta.
\]
 Let $N^\eta(x)=\{p \in \R^{n-1} \colon  H^\eta(x,p) \leq 0\}$.
 By repeating the same analysis as in the above proof in a careful manner, for $\eta>0$ such that  $\eta(\log 2 - \log r)\leq (1-\theta)\bar \delta$,
 \begin{equation*}
\theta N(x) \subset N^\eta(x).
\end{equation*}
Define
\[
\varphi^\theta(x)=
\begin{cases}
0 \quad &\text{ for } x \in \target \cup D,\\
\theta(V_r(x)-\delta) \quad &\text{ for } x\in (\mathcal{B}^1 \cap X_r)\setminus D.
\end{cases}
\]
Then $D\varphi^\theta(x)=0$ for $x \in (\target \cup D)^\circ$, and if $V_r$ is differentiable at $x \in (\mathcal{B}^1 \cap X_r)\setminus D$ then
\[
D\varphi^\theta(x) = \theta DV_r(x) \in \theta N(x) \subset N^\eta(x).
\]
Thus, for $\eta>0$ small enough, $\varphi^\theta$ is a solution to \eqref{eq:SolTarget-eta-r}.
We yield $V^\eta_r \geq \varphi^\theta$.
Combine this with Proposition \ref{prop:subseq-conv} to imply, for any sequence $\{\eta_k\} \to 0$ such that $V^{\eta_k}_r \to \bar V$ uniformly on $X_r$,
\[
\bar V \geq \varphi^\theta.
\]
We get the desired result by letting $\theta \to 1$ and $\delta \to 0$ in this order.
\end{proof}

We immediately get the following corollary.
\begin{cor}\label{cor:gen}
Assume $n \geq 2$, and $Z= \ol{X \setminus \cB^1}$.
For each $\eta>0$ sufficiently small, pick $r_\eta>0$ so that $\lim_{\eta \to 0} r_\eta = \lim_{\eta \to 0} \eta \log r_\eta = 0$.
Let $V$ be the maximal Lipschitz continuous solution to \eqref{eq:SolTarget}.
For each $\eta>0$, let $V^\eta$ be the maximal Lipschitz continuous solution to \eqref{eq:SolTarget-eta-r} with $r=r_\eta$.
Then $V^\eta \to V$ locally uniformly on $X^\circ$ as $\eta \to 0$.
\end{cor}

The proof of Corollary \ref{cor:gen} follows exactly the same lines as that of Theorem \ref{thm:main2} and hence is omitted.
A key point that we need here is
\[
\lim_{\eta \to 0} \eta (\log 2 - \log r_\eta) = 0,
\]
and therefore, for each $\theta \in (1/2,1)$ and $\bar \del>0$, there exists $\eta_0>0$ such that
\[
\eta(\log 2 - \log r_\eta)\leq (1-\theta)\bar \delta \qquad \text{ for all } \eta \in (0,\eta_0).
\]

\section{Conclusions} \label{sec:con}
We have completely answered Question \ref{Q2} by Theorem \ref{thm:main2}.
Besides, Theorem \ref{thm:V-r-to-V} gives us that $V_r \to V$ locally uniformly in $X^\circ$, which shows that it is quite reasonable to consider the problems in a restricted simplex $X_r$ for $r>0$ sufficiently small and pass to the limit if necessary.
This is also of practical use.

For Question 1, we have only the affirmative answer when $n=2$, and this was proved earlier in \cite{S}.
The question is still open for $n\geq 3$.
As $H^\eta$ and $L^\eta$ have quite singular behavior near $\partial X$ for each $\eta>0$, it seems that one needs to study finer properties of $H^\eta, L^\eta$ in order to proceed further.
For example, it is not clear at all if $V^\eta$, the maximal solution to \eqref{eq:SolTarget-eta}, is globally  Lipschitz on $X$ or not.

Here, we have addressed Questions \ref{Q1}--\ref{Q2} concerning the target problem with a fixed target $Z=\ol{X\setminus \cB^1}$.
Other selection problems with different targets should be considered and analyzed.
Besides, selection problems for the source problem should also be studied  in the near future.


\end{document}